\documentclass[11pt]{amsart}      
\usepackage{amssymb}      
\usepackage{amsmath}      
\usepackage{amsthm}      
\usepackage{amsbsy}      
\usepackage{amscd}      
\usepackage[dvips]{graphicx}

\theoremstyle{plain}      
\newtheorem{theorem}{Theorem}[section]      
\newtheorem{lemma}{Lemma}[section]      
\newtheorem{corollary}[theorem]{Corollary}      
\newtheorem{proposition}{Proposition}[section]

\newtheorem{definition}{Definition}[section]          
\theoremstyle{remark}      
\newtheorem{remark}{Remark}[section]

\newcommand{\Q}{{\mathbb{Q}}}        
\newcommand{\Z}{{\mathbb{Z}}}   
   
\newcommand{\C}{{\mathbb{C}}}      
\newcommand{\R}{{\mathbb{R}}}  

\renewcommand{\S}{{\mathcal{S}}}  

\newcommand{\Aut}{{\rm{Aut}^+}}  
    \newcommand{\Out}{{\rm{Out}^+}}

\newcommand{\ro}{{\widetilde{\rho}}}

\begin{document}

\date{\today}

\title[Profinite Burnside-type groups]{Profinite completions of Burnside-type quotients  of surface groups}

\author[L.Funar]{Louis Funar}
\address{Institut Fourier,
Laboratoire de Mathematiques UMR 5582,   
Universit\'e Grenoble Alpes,
CS 40700, 
38058 Grenoble, France}
\email{louis.funar@univ-grenoble-alpes.fr} 
\author[P.Lochak]{Pierre Lochak}
\address{Centre de Math\'ematiques de Jussieu, 
Universit\'e Paris Pierre et Marie Curie  
4, place Jussieu,  
F-75252 Paris cedex 05, France}
\email{pierre.lochak@imj-prg.fr}

\maketitle 

\begin{abstract} 
Using quantum representations of mapping class groups we prove that profinite completions of Burnside-type surface group quotients are not virtually prosolvable, in general. Further, we construct infinitely many finite simple 
characteristic quotients of surface groups.

\vspace{0.2cm}
 
\end{abstract}

\section{Introduction and statements}

Let $\pi_g$ denote the fundamental group $\pi_1(S_g,p)$  of a closed orientable surface $S_g$ of genus $g$, based at a point $p\in S_g$. Recall that $\pi_g$ is a one-relator group with the presentation:
\[ \pi_g=\langle a_1,a_2,\ldots,a_g,b_1,b_2,\ldots,b_g | \;  [a_1,b_1]\cdots [a_2,b_2]\cdots [a_g,b_g]=1\rangle \]
Here the classes  $a_i,b_i$ are represented by non-separating simple closed loops on $S_g$ based at $p$.

We denote by $\Gamma_g$ the mapping class group of $S_g$. Further 
$\Gamma_g^1$ denotes  the mapping class group of the pair $(S_g,p)$, namely the group of isotopy classes of  orientation preserving homeomorphisms of $S_g$ fixing $p$. It is well-known that $\Gamma_g^1$ is isomorphic to the mapping class group of the punctured surface $S_g-\{p\}$. 
 By forgetting the marked point $p$ one obtains a 
surjective homomorphism $\Gamma_g^1\to \Gamma_g$ which is part of the Birman exact sequence: 
\[ 1\to \pi_g \to \Gamma_g^1\to \Gamma_g\to 1\]
 
The Dehn-Nielsen-Baer theorem states that 
the map associating to $\varphi\in \Gamma_g^1$ the automorphism $\varphi_*:\pi_1(S_g,p)\to \pi_1(S_g,p)$
provides an  isomorphism between $\Gamma_g^1$ and $\Aut(\pi_g)$ and induces an isomorphism 
$\Gamma_g\to \Out(\pi_g)$.
Furthermore, the diagram below is commutative: 
\[\begin{array}{ccccccccc}
 1& \to & \pi_g & \to & \Gamma_g^1&\to & \Gamma_g &\to& 1\\
  &  & \downarrow & & \downarrow & & \downarrow & & \\
 1 & \to & \pi_g & \to & \Aut(\pi_g) & \to &\Out(\pi_g) & \to & 1 \\
 \end{array}
 \] 
where the top horizontal line is the Birman exact sequence. 
 
If $M\subset \pi_g$ we denote by $M[n]$ the normal subgroup of $\pi_g$ generated by 
$\varphi_*(x^n)$, for all $x\in M$ and $\varphi_*\in\Aut(\pi_g)$. 
Note that $M[n]$ is the characteristic subgroup generated by the subset $M^n$ of $n$-th powers of 
elements in $M$.  

The {\em  Burnside-type group}  $B(\pi_g,n,M)$ is the quotient $\pi_g/M[n]$. 
Several choices for $M$ are particularly interesting.  An element $x\in \pi_g$ is called {\em primitive}  if it can be represented 
by a non-separating simple closed curve on $S_g$.  This is equivalent to saying (see \cite{Z}) that $x\in \pi_g$ 
can be mapped into one generator, say $a_1$, by some  automorphism $\varphi_*\in \Aut(\pi_g)$, where $a_1,\ldots,a_g,b_1,\ldots,b_g$ are the generators from the standard presentation above. 

The set of primitive classes of $\pi_g$ is then contained in the set $\S(S_g)$ of homotopy classes of 
simple closed curves on $S_g$. 
More generally, we set  $\S_n(S_g)$ for  the set of homotopy classes of closed curves on $S_g$ with at most 
$n$ self-intersections.

We denote by $\widehat{G}$ the profinite completion of a group $G$. 
We are concerned in this paper with how large the profinite completion of $B(\pi_g,n, M)$ could be.  
Our first result is: 

\begin{theorem}\label{prosolvable}
Let $g\geq 2$ and  $p\equiv 3 ({\rm mod}\; 4)$ a large enough prime. 
Then for every $m$ there exists some $d$ such that 
the group $\widehat{B(\pi_g,dp,M)}$ is not virtually prosolvable, if $M\subset\S_m(S_g)$. 
When $m=1$ then $d=1$. 
\end{theorem}
\begin{remark}
The result above also holds for large enough primes  $p\equiv 1({\rm mod \; 4})$, according to  Remark \ref{mod4}. 
An explicit $p_0$ such that the claim holds for all $p\geq p_0$ can be obtained from effective bounds in Lemma 
\ref{square}. Moreover, the claim holds for all primes $p < 10^4$, by a computer check of Lemma \ref{square}.    
\end{remark}

The proof shows that under these assumptions $\widehat{B(\pi_g,dp,M)}$ is neither solvable-by-finite nor finite-by-solvable. Our notational convention is that a finite-by-solvable group is an extension 
of a finite group by some solvable group, also 
called a virtually solvable group in the litterature.  

Zelmanov \cite{Z} considered the group $\widehat{\pi_g}/\langle M[n]\rangle$, where 
$\langle M^n\rangle$ is the closure in $\widehat{\pi_g}$ of the normal subgroup of $\widehat{\pi_g}$ generated by 
$M^n$. Problem 2 from \cite{Z} asked whether this group is solvable-by-finite, when $M$ denotes the set of primitive elements of $\pi_g$. The result above shows that 
this is not the case, in general: 

\begin{corollary}
Let $g\geq 2$ and  $p\equiv 3 ({\rm mod}\; 4)$, a large enough   prime.
The group $\widehat{\pi_g}/\langle M[n]\rangle$, where $M=\S(S_g)$ contains the set of primitive elements of $\pi_g$, is 
not virtually solvable. 
\end{corollary}
\begin{proof}
The surjective map $\pi_g\to  B(\pi_g,p,M)$ induces a surjective continuous homomorphism between the 
corresponding profinite completions $\widehat{\pi_g}\to  \widehat{B(\pi_g,p,M)}$. The kernel of the last 
map contains $M[n]$ and hence the closure $\langle M^n\rangle$ of the normal subgroup of $\widehat{\pi_g}$ generated by 
$M^n$. Therefore we have a surjective continuous map $\widehat{\pi_g}/\langle M[n]\rangle\to \widehat{B(\pi_g,p,M)}$. 
\end{proof}

Our method also provides a large supplies of  finite quotients for all intermediary subgroups:

\begin{corollary}\label{ubiquity}
Let $g\geq 2$ and $\Gamma$ be a group such that $\pi_g\subset \Gamma\subset \Gamma_g^1$. 
Then $\Gamma$ admits surjective homomorphisms onto infinitely many finite simple groups, 
for instance  $PSL(N,\mathbf F_q)$, where $N$ and $q$ are arbitrarily large. 

In particular, this holds  if $\Gamma$ is the fundamental group of a closed 3-manifold fibering over the circle 
with fiber a closed orientable surface of genus $g\geq 2$. 
\end{corollary}

Recall that a subgroup $H\subset G$ is {\em characteristic} if it is invariant by the action of 
the group ${\rm Aut}(G)$ of automorphisms of $G$. Further, the quotient $Q$ of $G$ is a {\em characteristic quotient} 
of $G$ if there is a surjective homomorphism $p:G\to Q$ whose kernel $\ker p$ is a characteristic subgroup of $G$.   
A consequence of an intermediary result obtained in the proof of Theorem \ref{prosolvable} is the following: 

\begin{theorem}\label{Lubotzky}
For $g\geq 2$ there exist infinitely many finite simple characteristic quotients of $\pi_g$. 
\end{theorem}
This answers a question of Lubotzky from (\cite{Lub}, sections 10 and 6.4).

The proof of the main results goes as follows. 
We consider the so-called quantum representations of the mapping class groups $\Gamma_g$ and $\Gamma_g^1$
depending on some root of unity of order $2p$. 
It was proved in \cite{F} that these representations have infinite image, for $p\geq 5$. 
The proof was simplified in \cite{Mas2} where explicit elements of infinite order were found. Further, in \cite{LW} the authors showed that the images of $\Gamma_g$  are topologically dense in the corresponding special unitary groups, when $p\geq 5$ is prime. On the other hand the matrices in the images have coefficients in a cyclotomic ring (see \cite{GM}). 
Eventually the restriction of scalars provides  Zariski dense discrete representations in semi-simple linear algebraic groups 
defined over $\Q$ whose images are contained in arithmetic groups of higher rank (see \cite{GM}).  
The aim of \cite{F2} and \cite{MRe} was to  construct quotients of 
$\Gamma_g$ which are simple finite groups of Lie type of arbitrary large rank.

Our strategy here is to consider the restriction of the quantum representations  from $\Gamma_{g}^1$ to 
the subgroup $\pi_g$. These representations were recently studied by Koberda and Santharoubane in \cite{KS}, where 
it is proved that they still have infinite images, while they factor through  the Burnside-type group $B(\pi_g,p, \S(S_g))$.
Our aim is to show that the restriction of scalars provides  Zariski dense discrete representations of $B(\pi_g,p, \S(S_g))$ in some semi-simple linear algebraic group defined over $\Q$ 
of higher $\R$-rank. 
The Nori-Weisfeiler approximation theorem (see \cite{No,Weis}) then provides  
many finite quotients of congruence type. This implies that our profinite Burnside-type groups surjects onto 
an infinite product of simple non-abelian groups, proving our first theorem.  We note that 
the image of the surface group coincides with that of the mapping class group. 
Thus every kernel of a homomorphism of $\pi_g$ onto a finite simple quotient obtained this way is invariant by the 
mapping class group action. We obtain therefore finite simple characteristic quotients of $\pi_g$, 
proving the second theorem. 
Eventually, we notice that the quotients obtained by this method are 
principal  congruence quotients.

{\bf Acknowledgements}. The authors are indebted to C. Blanchet, S. Checcoli, P. Eyssidieux, T. Koberda, J. March\'e, B. Remy and T. Venkataramana  for useful discussions and to the referee for several corrections, suggestions improving the presentation and for pointing out that our methods also  apply to Lubotzky's question above.

\section{Preliminaries on quantum mapping class group representations}
%%%%%%%%%%%%%%%%%%%%%%%%%%%%%%%%%%%%%%
\subsection{The setting of the skein TQFT}\label{tqft}
%%%%%%%%%%%%%%%%%%%%%%%%%%%%%%%%%

 A TQFT   is a 
functor from the category of surfaces into the category of finite dimensional vector spaces. 
Specifically, the objects of the first category are closed oriented surfaces endowed with 
colored banded points and morphisms between two objects are cobordisms 
decorated by uni-trivalent ribbon graphs compatible with the banded points.
A banded point on a surface is a point with a tangent vector at that point, or equivalently 
a germ of an oriented interval embedded in the surface. There is a corresponding  
surface with colored boundary obtained by deleting a small neighborhood of the 
banded points and letting the boundary circles inherit the colors of the respective points.

We will use the TQFT functor $\mathcal V_p$, 
for $p\geq 3$ and a primitive root of unity $A$ of order $2p$, as 
defined in \cite{BHMV}. 
The vector space  associated by the functor $\mathcal V_p$ to a surface 
is called the {\em space of conformal blocks}.  Let $S_g$ denote the genus $g$ closed orientable surface, 
$H_g$ be a genus $g$ handlebody with $\partial H_g=\Sigma_g$. 
 Assume given a finite set  $\mathcal Y$ of banded 
points on $S_g$. Let $G$ be a uni-trivalent ribbon graph embedded in $H_g$ in such a way that 
$H_g$ retracts onto $G$, its univalent vertices are the banded points $\mathcal Y$ and it has no other intersections 
with $S_g$. 

We fix a natural odd number $p\geq 3$, called the {\em  level} of the TQFT. We define the 
{\em set of colors} in level $p$ to be $\mathcal C_p=\{0,2,4,\ldots,p-3\}$.

An edge coloring of $G$ is called {\em $p$-admissible} if the triangle inequality is 
satisfied at any trivalent vertex of $G$ and the sum of the three colors around a 
vertex is bounded by $2(p-2)$. 

Fix a coloring of the banded points $\mathcal Y$. Then there exists a basis of the space of conformal blocks associated to the surface $(\Sigma_g, \mathcal Y)$ with the  colored banded points (or the corresponding surface with colored boundary) 
which is indexed by the set of all $p$-admissible colorings of $G$ 
extending the boundary coloring. We  denote by $W_{g, (i_1,i_2,\ldots,i_r)}$ the vector space associated to the closed surface $\Sigma_g$ with $r$ banded points colored by $i_1,i_2,\ldots,i_r\in \mathcal C_p$. Note that 
banded points colored by $0$ do not contribute. 

Observe that an admissible $p$-coloring of $G$ provides an element of the skein module 
$S_{A}(H_g)$ of the handlebody  with banded boundary points colored $(i_1,i_2,\ldots,i_r)$, evaluated at the  
primitive $2p$-th root of unity $A$. This skein element is obtained by cabling the edges of $G$ by the Jones-Wenzl idempotents prescribed by the coloring and having banded points colors fixed.  
We suppose that $H_g$ is embedded in a standard way into the $3$-sphere $S^3$, so that the closure of 
its complement is also a genus $g$ handlebody $\overline{H}_g$.  There is then a 
sesquilinear form: 
\[ \langle \;,\; \rangle: S_{A}(H_g)\times S_{A}(\overline{H}_g)\to \C\]
defined by 
\[ \langle x, y \rangle= \langle x \sqcup y \rangle.\]
Here $x\sqcup y$ is the element of $S_{A}(S^3)$ obtained by the disjoint union of  $x$ and $y$ in 
$H_g\cup\overline{H}_g=S^3$, and $\langle \; \rangle: S_{A}(S^3)\to \C$ is the Kauffman bracket invariant. 

Eventually the space of conformal blocks $W_{g,(i_1,i_2,\ldots,i_r)}$ is the quotient 
$S_{A}(H_g)/\ker \langle\;,\; \rangle$ by the left kernel of the sesquilinear form above. It follows that $W_{g,(i_1,i_2,\ldots,i_r)}$ is endowed with an induced {\em Hermitian form} $H_{A}$. 

The projections of skein elements associated to the $p$-admissible colorings of a trivalent graph $G$ as above form an orthogonal basis of $W_{g,(i_1,i_2,\ldots,i_r)}$ with respect to $H_{A}$. It is known (\cite{BHMV}) 
that $H_{A}$ only depends on the $p$-th root of unity $\zeta_p=A^2$ and that in this orthogonal basis the diagonal entries belong to the totally real maximal subfield $\Q(\zeta_p+\overline{\zeta_p})$ (after rescaling). 

Let $G'\subset G$ be a uni-trivalent subgraph  whose degree one vertices are colored, corresponding to a 
sub-surface $\Sigma'$ of $\Sigma_g$ with colored boundary. The projections in $W_{g,(i_1,i_2,\ldots,i_r)}$ of skein elements associated to 
the $p$-admissible colorings of $G'$ form an orthogonal basis of the space of conformal blocks 
associated to the surface $\Sigma'$ with colored boundary components. 

There is a geometric action of the mapping class groups of the handlebodies $H_g$ and $\overline{H}_g$ respectively
on their skein modules and hence on the space of conformal blocks. Moreover, these actions extend to a projective  
action $\rho_{g, p, (i_1,\ldots,i_r),A}$ of $\Gamma_g^r$ on $W_{g,(i_1,i_2,\ldots,i_r)}$ respecting the Hermitian form $H_{\zeta_p}=H_A$. When referring to $\rho_{g, p, (i_1,\ldots,i_r),A}$ the subscript 
specifying the genus $g$ will most often be dropped when its value will be clear from the context.
Notice that the mapping class group of an essential (i.e. without annuli or disks complements) 
sub-surface $\Sigma'\subset \Sigma_g$ is a subgroup of $\Gamma_g$ which preserves the subspace of conformal blocs 
associated to $\Sigma'$ with colored boundary.  It is worthy to note that  $\rho_{p,(i_1,\ldots,i_r), A}$ only depends on 
$\zeta_p=A^2$, so we can unambigously shift the notation for this representation to $\rho_{p,(i_1,\ldots,i_r),\zeta_p}$. 

There is a central extension $\widetilde{\Gamma_g}$ of $\Gamma_g$ by $\Z$ and a linear representation  
 $\ro_{p, \zeta_p}$ on $W_{g}$ which  resolves the projective ambiguity of $\rho_{p,\zeta_p}$. 
 The largest such central extension has class $12$ times the Euler class (see \cite{Ger,MR}), but 
 the central extension considered in this paper is  an index $12$ subgroup of it, called $\widetilde{\Gamma}_1$ 
 in \cite{MR}. When $g\geq 3$ it is a perfect group which therefore coincides with the universal central extension. 
 
 We denote by $S_{g,n}^r$ the compact orientable surface of genus $g$ with $n$ boundary components and $r$ marked points. Then $\Gamma_{g,n}^r$ denotes the pure mapping class group of $S_{g,n}^r$ which fixes pointwise boundary components and marked points. 
 
We consider a subsurface $\Sigma_{g,r}\subset \Sigma_{g+r}$ whose complement consists of $r$ copies of $\Sigma_{1,1}$. 
Let $\widetilde{\Gamma_g^r}$ be the pull-back of the central extension $\widetilde{\Gamma_g}$ to 
the subgroup $\Gamma_{g,r}\subset \Gamma_{g+r}$. Then $\widetilde{\Gamma_{g,r}}$ is also a central extension, which we denote $\widetilde{\Gamma_{g}^r}$ of $\Gamma_g^r$ by $\Z^{r+1}$. 
From \cite{Ger,MR} we derive that $\widetilde{\Gamma_g^r}$ is perfect, when $g\geq 3$ and of 
order $10$, when $g=2$.

\begin{definition}\label{qrep}
Let $p\geq 5$ be odd and $\zeta_p$  a primitive 
$p$-th root of unity. We denote 
by $\ro_{p, \zeta_p, (i_1,i_2,\ldots,i_r)}$ the linear representation of the central extension 
$\widetilde{\Gamma_g^r}$ which acts on the vector space $W_{g,p,(i_1,i_2,\ldots,i_r)}$ associated by the TQFT to the surface with the corresponding colored banded points (see \cite{Ger,MR}).  
\end{definition}

The functor $\mathcal V_p$ associates to a handlebody $H_g$ the projection of the skein element 
corresponding to the trivial coloring of the trivalent graph $G$ by $0$. The invariant associated 
to a closed 3-manifold is given by pairing the two vectors associated to handlebodies in a Heegaard decomposition 
of some genus $g$ and taking into account the twisting by the gluing mapping class action on $W_g$. 

One should notice that the skein TQFT $\mathcal V_p$ is unitary, in the sense that $H_{\zeta_p}$ is 
a positive definite Hermitian form when $\zeta_p=(-1)^p \exp\left(\frac{2\pi i}{p}\right)$, corresponding to  
$A_p=(-1)^{\frac{p-1}{2}}\exp\left(\frac{(p+1)\pi i}{2p}\right)$. 
For the sake of notational simplicity, from now we will drop the subscript $p$ 
in $\zeta_p$, when the order of the root of unity will be clear from the context and the precise choice of the root of given order won't matter.  
Note that for a general primitive $p$-th root of unity, 
the isometries of $H_{\zeta}$ form a pseudo-unitary group.

Now, the image $\rho_{p,\zeta}(T_{\gamma})$  of a right hand Dehn twist $T_{\gamma}$  in 
a convenient basis given by a trivalent graph is easy to describe. 
Assume that the simple curve $\gamma$ is the boundary of a small disk intersecting once transversely 
an edge  $e$ of the graph $G$. Consider $v\in W_g$ be a vector of the basis given by colorings of the graph $G$ 
and assume that edge $e$ is labeled by the color $c(e)\in \mathcal C_p$. Then the action of 
the (canonical)  lift  $\widetilde{T_{\gamma}}$ of the Dehn twist $T_{\gamma}$ in $\widetilde{\Gamma_g}$ 
is given by (see \cite{BHMV}, 5.8) :
\[ \ro_{p,\zeta}(\widetilde{T_{\gamma}}) v =A^{c(e)(c(e)+2)} v\]

%%%%%%%%%%%%%%%%%%%%%%%%%%%%%%%%%%%%%%
\subsection{Unitary groups of spaces of conformal blocks}
%%%%%%%%%%%%%%%%%%%%%%%%%%%%%%%%%%%%%%
%%%%%%%%%%%%%%%%%%%%%%%%%%%%%%%%%%%%%%

For a prime $p\geq 5$  we denote 
by ${\mathcal O}_p$ the  ring of cyclotomic integers 
${\mathcal O}_p=\Z[\zeta_p]$, if   
$p\equiv 3({\rm mod}\: 4)$ and ${\mathcal O}_p=\Z[\zeta_{4p}]$, if  
$p\equiv 1({\rm mod}\:4)$ respectively, where $\zeta_r$ denotes a primitive
$r$-th root of unity (the subscript $r$ will sometimes be omitted). 
The main result of \cite{GM} states that there exists a free ${\mathcal O}_p$\,-lattice 
$\Lambda_{g,p}$ in the $\C$-vector space of conformal 
blocks associated by  the TQFT ${\mathcal V}_p$ to the genus $g$
closed orientable surface and a non-degenerate Hermitian  
${\mathcal O}_p$-valued form on 
$\Lambda_{g,p}$ both invariant under  the action of $\widetilde{\Gamma_g}$ via the representation 
$\widetilde{\rho}_{p,\zeta}$. 
Therefore the image of the mapping class group consists of 
unitary matrices (with respect to the Hermitian form) with 
entries in ${\mathcal O}_p$. 
Let  $\mathbb U_{g,p, \zeta}({\mathcal O}_p)$ and  $P\mathbb U_{g,p,\zeta}({\mathcal O}_p)$ be the group of all such matrices and respectively its quotient by scalars. 
Then $\mathbb U_{g,p, \zeta}({\mathcal O}_p)$ is the group of ${\mathcal O}_p$-points of 
the unitary group $\mathbb U_{g,p, \zeta}$ associated to the Hermitian form $H_{\zeta}$, which is 
a linear algebraic group defined over $\Q(\zeta+\overline{\zeta})$. 

When $p$ is  prime $p\geq 5$ and $g\geq 2$, $(g,p)\neq 5$, then  
$\ro_{p,\zeta_p}$ takes values in  the special unitary group $S\mathbb U_{g,p,\zeta_p}$ (see \cite{DW,
F2,FP2}).  
It is known that $S\mathbb U_{g,p,\zeta}({\mathcal O}_p)$ is an irreducible 
lattice in a semi-simple algebraic group $\mathbb G_{g,p}$ obtained 
by the so-called restriction of scalars construction from the 
totally real cyclotomic field $\Q(\zeta_p+\overline{\zeta_p})$ to $\Q$. 
Specifically, the group  $\mathbb G_{g,p}$ is a product 
$\prod_{\sigma\in S(p)}S\mathbb U_{g,p,\sigma(\zeta)}$. Here $S(p)$ stands for 
a set of representatives of the classes of complex 
embeddings $\sigma$ of $\mathcal O_p$ 
modulo complex conjugation, or equivalently the set of places of 
the totally real  cyclotomic field $\Q(\zeta_p+\overline{\zeta_p})$. 
The factor $S\mathbb U_{g,p,\sigma(\zeta)}$ is the 
special unitary group associated to the 
Hermitian form conjugated by $\sigma$, thus corresponding to some 
Galois conjugate root of unity.  

Denote   
by $\ro_p$ and $\rho_p$ the representations  
$\prod_{\sigma\in S(p)} \ro_{p,\sigma(A_p^2)}$ and 
$\prod_{\sigma\in S(p)} \rho_{p,\sigma(A_p^2)}$, respectively. 
Notice that the real Lie group $\mathbb G_{g,p}$ 
is a semi-simple algebraic group defined over $\Q$.

In \cite{F2} it is proved that 
$\ro_p(\widetilde{\Gamma_g})$ is a discrete Zariski dense subgroup 
of $\mathbb G_{g,p}(\R)$ whose projections onto the simple factors of 
$\mathbb G_{g,p}(\R)$ are topologically dense, for $g \geq 3$ and  $p\geq 7$ 
prime, $p\equiv 3({\rm mod }\; 4)$. 

\begin{remark}\label{mod4}
When $p\equiv 1 ({\rm mod }\; 4)$ the image of the central extension 
of $\Gamma_g$ from \cite{MR} by $\ro_p$ is contained in 
$\mathbb G_{g,p}(\Z[i])$ and thus it is a discrete 
Zariski dense subgroup of $\mathbb G_{g,p}(\C)$. 
However, if we restrict to the universal central extension  
$\widetilde{\Gamma_g}$  coefficients are reduced from 
$\Z[\zeta_{4p}]$ to $\Z[\zeta_p]$ (see  \cite{GM}, section 13). 
Note that the corresponding invariant form $H_{\zeta_p}$ should be suitably 
rescaled and after rescaling it will be skew-Hermitian when $g$ is odd and Hermitian for even $g$.  

As mentioned in (\cite{MRe}, Rem.3.5) for the proof of our main result we don't need the integral TQFT of \cite{GM} as the 
Burnside-type groups are finitely generated and hence only finitely primes could appear in the denominators of 
matrices in their image. 
\end{remark}

\section{ Quantum surface group representations}
\subsection{Zariski density of quantum representations} 
Our aim is to find the Zariski closures of $\rho_{p, (i)}(\pi_g)$. 
We follow closely the strategy from \cite{F2}, where we proved that 
$\rho_{p, (i)}(\Gamma_g)$ is Zariski dense in $P\mathbb G_p(\R)$, based on the topological density result 
in the corresponding special unitary group earlier obtained in 
\cite{LW}.

The mapping class group $\Gamma_{g,1}$ is a subgroup of $\Gamma_{g+1}$, by identifying 
$S_{g+1}$ with the result of gluing of $S_{g,1}$ and $S_{1,1}$. 
It is well-known that 
\[ W_{g+1,p} =\bigoplus _{i} W_{g,p,(i)}\otimes W_{1,p,(i)}\]
The decomposition corresponds to the eigenspaces for the Dehn twist $\widetilde{T_c}$ along the curve $c=\partial S_{g,1}$.  
Let $\mathbb U_{g,p,\zeta, (i)}=U(W_{g,p,(i)},H_{\zeta})$ be the unitary subgroup keeping invariant the subspace $W_{g,p,(i)}$, when endowed 
with the (restriction of the)  Hermitian form $H_{\zeta}$. The group $\mathbb U_{g,p,\zeta, (i)}$ is a 
closed  linear algebraic  subgroup of $\mathbb U_{g+1,p,\zeta}$ and is also defined over the 
maximal totally real algebraic field $\Q(\zeta+\overline{\zeta})$ of $\Q(\zeta)$.  

Since $\widetilde{\Gamma_{g}^1}$ 
is perfect when $g\geq 3$ and of order $10$ for $g=2$, it follows that  $\ro_{p, (i)}(\widetilde{\Gamma_g^1})$ 
is contained within the special unitary group $S\mathbb U_{g,p,\zeta, (i)}$, if $(g,p)\neq (2,5)$, as in \cite{DW,F2,FP2}.

We denote by $\mathbb G_{g,p,(i)}$ the group obtained by scalar restriction from $\Q(\zeta+\overline{\zeta})$ to $\Q$ 
of the linear algebraic group $S\mathbb U_{g,p,\zeta, (i)}$, namely the product 
$\mathbb G_{g,p,(i)}=\prod_{\sigma\in S(p)}S\mathbb U_{g,p,\sigma(\zeta), (i)}$.  
It follows that the product representation  $\ro_{p,(i)}=\prod_{\sigma\in S(p)} \ro_{p, \sigma(A_p^2), (i)}$ of $\widetilde{\Gamma_g^1}$ takes values in $\mathbb G_{g,p,(i)}$. 
Since the boundary Dehn twist acts as a scalar this  representation of $\widetilde{\Gamma_g^1}$ descends to a 
projective representation $\rho_{g,p,(i)}:\Gamma_{g}^1\to P\mathbb G_{g,p,(i)}$. 

Set $\widetilde{\pi_g}=\ker(\widetilde{\Gamma_g^1}\to \Gamma_g)$. It follows that $\widetilde{\pi_g}$ is an extension by $\Z^2$ of $\pi_g$.

Our main result in this section is: 
\begin{theorem}\label{Zdensity} 
Let $g\geq 2$ and  $p\equiv 3 ({\rm mod}\; 4)$, $p$ a large enough prime.
Then the Zariski closure of $\ro_{p, (p-3)}(\widetilde{\pi_g})$ is $\mathbb G_{g,p,(p-3)}(\R)$. 
Moreover, if $g\geq 3$ every non-compact factor of $\mathbb G_{g,p,(p-3)}(\R)$ 
has real rank at least $2$.   
\end{theorem}

The rest of this section is devoted to the proof of the theorem above. 

The key ingredient is the following proposition, whose rather technical proof is postponed to 
section \ref{proofdensity} below: 
\begin{proposition}\label{mcgdensity} 

Let $g\geq 2$ and  $p\equiv 3 ({\rm mod}\; 4)$, $p$ a large enough prime.
The representation $\ro_{p, \zeta, (p-3)}$ of $\widetilde{\Gamma_{g,1}}$ into $W_{g,p,(p-3)}$ has dense image in the special unitary group $S\mathbb U_{g,p,\zeta, (p-3)}$.   
\end{proposition}

\begin{proposition}\label{density}
Let $g\geq 2$ and $p\geq 5$ be odd. 
Suppose that $\ro_{p, \zeta, (i)}(\widetilde{\Gamma_{g,1}})$ is Zariski dense in  $S\mathbb U_{g,p,\zeta, (i)}$. 
Then  $\ro_{p, \zeta, (i)}(\widetilde{\pi_{g}})$ is Zariski dense 
in the special unitary group $S\mathbb U_{g,p,\zeta, (i)}$.   
\end{proposition}
\begin{proof}[Proof of Proposition \ref{density}]
As $\pi_g$ is a normal subgroup of $\Gamma_g^1$, we derive that the topological closure of its image by $\rho_{p,\zeta, (i)}$ is a 
closed normal Lie subgroup of the projective unitary group $P\mathbb U_{g,p,\zeta,(i)}$. 
Therefore the image of $\widetilde{\pi_g}$ is a closed normal subgroup 
of $S\mathbb U_{g,p,\zeta,(i)}$. Since the Lie algebra of $S\mathbb U_{g,p,\zeta,(i)}$ is simple 
it follows that the Lie group has dimension zero and hence it is a discrete subgroup. However a normal discrete subgroup 
of $S\mathbb U_{g,p,\zeta,(i)}$ must be contained in its center, which is cyclic of order 
$\dim W_{g,p,(i)}$. 

Now, the result of \cite{KS} for $i=2$ shows that the image of $\pi_g$  by $\rho_{p,\zeta, (i)}$ is infinite non-abelian. 
We claim that this holds true for all $i\neq 0$ and we  will give a detailed proof for $i=p-3$. 

The $k+1$-holed sphere $S_{0,k+1}$, whose boundary circles are colored by  
$\mathbf c=(a,a,\ldots,a, ak-2)$ has associated a space of conformal blocks $W_{0,\mathbf c}$ of dimension $k$
which has a natural action of the braid group $B_{k}$ on $k$ strings.  Note that $\Sigma_{0,k+1}$  
can be embedded into $\Sigma_{g,1}$ such that the homomorphism $B_{k}\to  \Gamma_{g,1}$ is injective, if $k\leq g$.  
It is well-known that this braid group action coincides 
with the Burau representation at a suitable root of unity (see \cite{FK}) 
twisted by a character. Specifically, the Burau representation is 
the one for which the standard braid generators have eigenvalues 
$-1$ and $A_p^{2a^2}$. Moreover, in \cite{FK} one proved that the image of the Burau representation 
of $B_3$ is infinite non-abelian if $A_p^{2a^2}$ is not a primitive root of unity of order $2,3,4$ or $5$, while 
the image of $B_4$ is  infinite non-abelian  (see e.g. \cite{F}) if  $A_p^{2a^2}$ is not a primitive root of unity of 
order $2$ or $3$. 

It suffices to consider $i\geq 4$. If we can write $i=ak-2$, $3\leq k \leq g$, $a\in \mathcal C_p$, 
the image of $B_k$ is infinite.  Further $\pi_1(S_{0,3})\subset \pi_1(S_{0,k+1})$, if $k\geq 2$ and the restriction of the Burau representation to  the pure braid group $PB_3$ is infinite non-abelian. But $PB_3=\mathbb F_2 \times \Z$, where the factor $\Z$ is central and its image by the Burau representation is of finite order, while the free factor $\mathbb F_2$ can be identified with $\pi_1(S_{0,3})$. We derive that the image of $\pi_1(S_{g,1})$ by the subrepresentation of $\rho_{p, (i)}$ 
corresponding to the Burau representation contains the image of $\mathbb F_2$, namely a triangle group according to \cite{FK}.

When  $g=2$ and $i=p-3$  we consider the image of $\pi_1(\Sigma_{1,2})$ by 
the quantum representation of $\Gamma_{2}^1$, where $\Sigma_{1,2}\subset \Sigma_{2,1}$ is the complementary of a one holed torus with boundary label $2$. 
Then $\Gamma_{1,2}$ acts on $W_{1,p, (2,p-3)}$ which has dimension $3$ and 
an explicit calculation shows that the image of $\pi_1(\Sigma_{1,2})$ is infinite non-abelian. 
 \end{proof}

Let now $\Gamma\to H_i$, $i=1,\ldots,m$, 
be a collection of representations of the group $\Gamma$.  
The subgroup $H\subset \prod_{i=1}^mH_i$ is called 
$\Gamma$-diagonal, if there exists a partition $A_1,\ldots,A_s$ of 
$\{1,2,\ldots,m\}$ such that: 
\begin{enumerate}
\item All factors $H_i$, with $i\in A_t$, $1\leq t\leq s$ are 
equivalent as  representations of $\Gamma$.  Pick up some $i_t\in A_t$.  
Given some  intertwining isomorphisms 
 $L_{j,i_t}: H_j\to H_{i_t}$,  $j\in A_t\setminus\{i_t\}$, we set:  
\[H_{A_t}=\{(x,  
(L_{j,i_t}(x))_{j\in A_t\setminus\{i_t\}}), \;\;  
x\in H_{i_t}\},\]    
which is the graph of  the 
homomorphism $\oplus_{j\in A_t\setminus\{i_t\}} L_{j,i_t}$.
\item Then there exist intertwining isomorphisms  as above with the property 
that the group $H$ contains $\prod_{1\leq t\leq s} H_{A_t}$. 
In particular, if all representations $H_i$ of $\Gamma$ are 
pairwise inequivalent, then 
$H=\prod_{i=1}^mH_i$.
\end{enumerate}

We then have the following Hall lemma from \cite{Ku}: 

\begin{lemma}[\cite{Ku}]\label{Hall-lemma}
Let $\Gamma$ be a subgroup of  the product $\prod_{i=1}^mH_i$ of the 
adjoint simple (i.e. connected, without center and whose Lie algebra 
is simple) Lie groups $H_i$.  
Assume that the projection of $\Gamma$ on each factor $H_i$ 
is Zariski dense. Then the Zariski closure of $\Gamma$ 
in  $\prod_{i=1}^mH_i$ is a $\Gamma$-diagonal subgroup.  
\end{lemma}

We will use next the following classical result of Dieudonn\'e (\cite{Di}) and Rickart (\cite{Ri}, Thm. 4.3):  
\begin{proposition}[\cite{Di,Ri}]\label{Dieu-Rick}
Any group isomorphism $L: U(W_1)\to U(W_2)$ between the unitary groups of Hermitian 
vectors spaces $W_1$ and $W_2$ has the form: 
\[ L(h)  =\chi(h)\cdot V^{-1} h V\]   
where the map   $V:W_1\to W_2$  is either 
linear or anti-linear, and $\chi: U(W_1)\to U(1)$ is a homomorphism. 
\end{proposition}

\begin{lemma}
Let $A$ and $B$ be primitive $2p$-th roots of unity, for odd $p$. 
 If $\ro_{p, A^2, (p-3)}|_{\widetilde{\pi_g}}$ and $\ro_{p, B^2, (p-3)}|_{\widetilde{\pi_g}}$ are linearly or anti-linearly equivalent, then either $A=B$ or $A=\overline{B}$.  
\end{lemma}
\begin{proof}
According to Proposition \ref{Dieu-Rick}  the two representations in the same unitary group $U(W)$ 
are equivalent only  if there exists an intertwiner  (either linear or anti-linear) map  $V:W\to W$ which conjugates 
the two representations, possibly up to twisting by a character $\chi: U(W)\to U(1)$. 
In our case the representations take values into the special unitary group and hence we can take $\chi=1$. 
 
Observe that $V$ should send eigenspaces for 
$\ro_{p,A^2,(i)}(\gamma)$ to eigenspaces for $\ro_{p,B^2,(i)}(\gamma)$ of the same eigenvalues. 
If $\gamma$ a simple non-separating based loop on the surface, let $\gamma_+,\gamma_-$ denote 
the curves obtained by slightly pushing left and right respectively. Then $\gamma_+,\gamma_-$ and a small 
circle around the base point determine a pair of pants $S_{0,3}$ whose complement  $S_{g-1,2}$ 
is a genus $g-1$ surface with two boundary components. Therefore 
\[ \ro_{p,A^2,(i)}(\widetilde{\gamma})x=A^{j(j+2)-k(k+2)} x, \; {\rm if} \; x\in W_{0,(i,j,k)}\otimes W_{g-1,(j,k)}\]
where the lift $\widetilde{\gamma}\in \widetilde{\pi_g}$ is given by 
$\widetilde{T_{\gamma_+}}\widetilde{T_{\gamma_-}}^{-1}\in \widetilde{\Gamma_g^1}$. 

It follows that $V$ should send vector spaces of the form 
$W_{0,(i,j,k)}\otimes W_{g-1,(j,k)}$ into spaces of the same form associated to possibly different labels.

Consider $i=p-3$. Therefore,  the only possibilities for $j,k$ such that $\dim W_{0,(i,j,k)}$ be non-zero is 
$j=p-3-2m, k=2m$, for some $2m\in \mathcal C_p$. 
Now, observe that the symmetry exchanging the two boundary components induces an isomorphism 
between $W_{g-1,(j,k)}$ and $W_{g-1,(k,j)}$. Further, consider a circle embedded in  $S_{g-1,2}$ 
which bounds a pair of pants along with the two boundary circles.  If $\ell$ is a label for the third circle then 
the set of $p$-admissible $\ell$ for boundary labels $(j,k)$, where $j>k$ is  strictly contained 
in the set of $p$-admissible values of $\ell$ for the boundary labels $(j-2,k+2)$. 
It follows that $\dim W_{g-1,(j,k)}$ are distinct for  all values $j\geq k$, with $j+k=p-3$. 

Therefore either $V$ keeps invariant each subspace  $W_{0,(i,j,k)}\otimes W_{g-1,(j,k)}$ or else 
$V$ sends every $W_{0,(i,j,k)}\otimes W_{g-1,(j,k)}$ onto 
$W_{0,(i,k,j)}\otimes W_{g-1,(k,j)}$. 
Since the corresponding eigenvalues should be the same we derive that 
either $A=B$ or $A=\overline{B}$. 
\end{proof}

{\em End of the proof of Theorem \ref{Zdensity}}. 
The Hall Lemma \ref{Hall-lemma}  shows that the Zariski closure of 
$\rho_{p,(p-3)}(\pi_g)$ is all of $\mathbb P\mathbb G_{p, (p-3)}(\R)$.
Now, using (\cite{Ku}, Lemma 3.6) we obtain that 
$\ro_{p,(p-3)}(\widetilde{\pi_g})$ is Zariski dense in $\mathbb G_{p,(p-3)}(\R)$.  

Finally notice that  $\mathbb G_{g,p,(p-3)}$ contains $\mathbb G_{g-1,p}$  as a subgroup. 
In particular, for $g\geq 4$ each non-compact factor has  rank at least $2$, by \cite{FP2}. 
We can follow the proof of this result in \cite{FP2} for $i=0$ to obtain the result for $g=3$ as well. 
This proves the theorem.

\subsection{Preliminaries on Verlinde formulas}
We start by collecting a few properties of  the dimensions of the space of conformal blocks. 
The main tool is the combinatorial description of the space of conformal blocks which admits a basis 
indexed by the set of $p$-admissible colorings of any uni-trivalent graph  associated to the surface, possibly with 
colored boundary components, as explained in 2.1. As a consequence, if we split a surface  
$S_{g,k}$ by cutting along $r$ essential pairwise non isotopic simple curves into the subsurfaces $S_{h, s+r}$ and  $S_{g-h-r+1, k-s+r}$
then we have a corresponding decomposition for the spaces of conformal blocks: 
\[ W_{g,p, (i_1,\ldots, i_k)}=\sum_{j_1,\ldots j_r\in \mathcal C_p} W_{h, p, (i_1,\ldots, i_s, j_1,\ldots, j_r)}
\otimes W_{g-h-r+1, p, (i_{s+1},\ldots, i_s, j_1,\ldots, j_r)}\]

\begin{lemma}\label{genone}
 \[ \dim W_{1,p,(j,i)}=\frac{(p-1-\max(i,j))(\min(i,j)+1)}{2}\]
 \end{lemma}
 \begin{proof}
Direct computation using the combinatorial description of the vector space. 
\end{proof}

\begin{lemma}\label{gentwo}
For $k\in \mathcal C_p$ we have: 
\[ \dim W_{2,p (k)}=
\frac{1}{24}\cdot \left((k+1)p^3-\frac{3}{2}k(k+2)p^2+\frac{1}{2}(k^3+3k^2-4)p\right)\]
and, in particular: 
\[ \dim W_{2, p,(p-3)}= \frac{p^3-p}{24}\]
\end{lemma}
\begin{proof}
We have 
\[ \dim W_{2,p (k)}=\sum_{j\in \mathcal C_p}\dim W_{1,p, (j)} \cdot \dim W_{1,p, (j,k)}\]
then expand all terms using Lemma \ref{genone}. 
\end{proof}

\begin{lemma}\label{genthree}
\[ \dim W_{3, p, (p-3)}=\frac{1}{5760}\cdot p(p-1)(p-3)(7p^3+28p^2+101p+80)+\frac{1}{24}\cdot (p^3-p)\]
\end{lemma} 
\begin{proof}
This is a consequence of Lemmas \ref{genone} and \ref{gentwo} along with 
\[ \dim W_{3, p, (p-3)}=\sum_{j\in \mathcal C_p}\dim W_{2, p, (j)} \cdot \dim W_{1,p, (j,p-3)}\]
\end{proof}

\begin{lemma}\label{compare}
We have $\dim W_{g, p, (p-3)} >\dim W_{g, p, (0)}$, if $g \geq 3$. 
\end{lemma}
\begin{proof}
We will prove by induction on $g$ that 
 $\dim W_{g, p, (k)} \geq \dim W_{g, p, (0)}$, for any $k\in \mathcal C_p$, with equality only if 
 $k=0$, $g\geq 3$ or $g=2$ and $k=p-3$. 
 
When $g=2$ the explicit formula from Lemma \ref{gentwo} allows for a direct verification. 
Assume that our claim holds true for all genera up to $g$.  We can write from Lemma \ref{genone}:  
\[ \dim W_{g+1,p, (k)}=\sum_{j\in \mathcal C_p} \dim W_{g,p, (j)} \cdot \frac{(p-1-\max(k,j))(\min(k,j)+1)}{2}\]
to be compared with 
\[ \dim W_{g+1,p, (0)}=\sum_{j\in \mathcal C_p} \dim W_{g,p, (0)} \cdot \frac{(p-1-j)}{2}\]
Now the induction hypothesis $\dim W_{g+1, p,(0)} \leq \dim W_{g,p,(j)}$ for all $j\in \mathcal C_p$ 
implies the claim for $g+1$. 
\end{proof}

\begin{lemma}\label{growth}
For any $g\geq 3$, $p\geq 7$ we have 
\[ \dim W_{g+1, p, (p-3)} < \frac{\dim W_{g, p,(p-3)}(\dim W_{g,(p-3)}-1)}{2}\]
Further, for $g=2$ we have the weaker inequality:
\[ \dim W_{3, p, (p-3)} < (\dim W_{2, p,(p-3)})^2\]
\end{lemma}
\begin{proof}
Recall the Verlinde formula (see \cite{GM2}) computing the dimension of the space of conformal blocks: 
\[ \dim W_{g, p,(k)}=\left(\frac{p}{4}\right)^{g-1}\sum_{s=1}^{\frac{p-1}{2}} \sin\left(\frac{(k+1)\pi s}{p}\right)
\sin\left(\frac{\pi s}{p}\right)^{1-2g}\]

Set $\alpha_s=\left(\frac{p}{4}\right)\sin\left(\frac{\pi s}{p}\right)$. 
If $g\geq 4$ we have the following inequalities: 
\[ \sum_{s}\alpha_s^{g} \leq  \sum_s \alpha_s^{4(g-1)/3}<(\sum_s \alpha_s)^{4/3}\]
which imply that:
\[ \dim W_{g+1, p, (0)} < (\dim W_{g,p, (0)}) ^{4/3}, \;  {\rm whenever} \;  g\geq 4\]

We derive from Lemma \ref{compare} that whenever $g\geq 4$ we have: 
\[ \dim W_{g+1, p, (p-3)} <\dim W_{g+2,p,(0)} \leq (\dim W_{g+1, p,(0)})^{4/3} < (\dim W_{g, p, (0)})^{16/9}< (\dim W_{g, p, (p-3)})^{16/9}\] 
On the other hand 
\[ (\dim W_{g, p, (p-3)})^{16/9} < \frac{\dim W_{g, p, (p-3)}(\dim W_{g, p, (p-3)}-1)}{2}\]
if $g\geq 4$ and $p\geq 5$, since $\dim W_{g, p, (p-3)}\geq \dim W_{4, 5, (2)}=75$.   

Eventually, we have to check the case when $g=3$.  From Lemma \ref{gentwo}
\[ \dim W_{2,p,(k)} < \frac{1}{24}\left((k+1)p^3 +\frac{k^3+3k^2}{2}p\right)< \frac{p^3(3p+5)}{48}\]
We have the following crude upper bound:
\[ \sum_{j\in \mathcal C_p}\frac{(p-1-\max(k,j))(\min(k,j)+1)}{2} < \frac{1}{4}p^3\] 
which leads to the upper bounds:
\[ \dim W_{3, p, (k)}< \frac{p^3(3p+5)}{48}\sum_{j\in \mathcal C_p}\frac{(p-1-\max(k,j))(\min(k,j)+1)}{2}<
\frac{p^6(3p+5)}{192}\]
and further
\[ \dim W_{4, p, (k)}< \frac{p^6(3p+5)}{192}\sum_{j\in \mathcal C_p}\frac{(p-1-\max(k,j))(\min(k,j)+1)}{2}<
\frac{p^9(3p+5)}{728}\]
Now, if $p > 35$ we have that 
\[ \dim W_{4, p, (p-3)} < \frac{p^9(3p+5)}{728} < \frac{7p^{12}}{32 \times (5760)^2} 
 < \frac{\dim W_{3,p,(p-3)}(\dim W_{3,p,(p-3)}-1)}{2}\]
The cases when $5\leq p \leq 35$ can be verified by  a direct computer search. 

Finally, the inequality claimed for $g=2$ is a consequence of Lemmas \ref{gentwo} and \ref{genthree}. 
 
Note that the inequality for $g\geq 3$ is actually valid with the same proof for all labels $i$ on the boundary circle. 
\end{proof}
\begin{remark}
The  inequality stated in Lemma \ref{growth} for $g\geq 3$ does not hold when $g=2$. Indeed we have the following asymptotical behavior, derived from Lemma \ref{genthree}: 
\[\lim_{p\to \infty} \frac{\dim W_{3, p, (p-3)}}{(\dim W_{2, p, (p-3)})^2} \simeq 0.7\]  
\end{remark}

\begin{lemma}\label{square}
There exist  only finitely many $p$ such that 
$1+ 8\dim W_{3,p,(p-3)}$ is a perfect square. 
\end{lemma}
\begin{proof}
The function $f(p)=1+ 8\dim W_{3,p,(p-3)}$ is a degree 6 square free polynomial in $p$. 
Faltings (see \cite{Fa}) proved that a non-singular algebraic curve of genus at least two 
which is defined over a number field has only finitely many rational points 
(Mordell's conjecture).
Since the projective curve given by the affine equation $y^2=f(x)$ is a hyperelliptic curve 
of genus $2$, the equation $y^2=f(x)$ has therefore only finitely many 
solutions in $\Q$.  

Alternatively, we can use a classical theorem of Siegel, which asserts that a smooth affine algebraic curve defined over $\Q$ of genus at least one has only finitely many points with integral coordinates (see \cite{Lang}, chap. VI). In particular, a 
polynomial with integer coefficients and at least 3 distinct roots takes only finitely many square values 
on the integers. Although $(24)^2(1+ 8\dim W_{3,p,(p-3)})$ has rational coefficients, by considering the change of variable 
$p=5q+s$, for each $s\in\{0,1,2,3,4\}$ we obtain five polynomials  with integer coefficients to each of which Siegel's theorem applies. 
\end{proof}

\subsection{Proof of Proposition \ref{mcgdensity}}\label{proofdensity}
Larsen and Wang in \cite{LW}  proved the topological density of the image of the mapping class group 
of a closed surface of genus $g$. This result corresponds to  the case when $i=0$ and $\zeta=A_p^2$. 
We will show that their proof suitably amended actually works for $i=p-3$ and $g\geq 2$. 
Some of the steps below are valid for every color $i$, but 
for the sake of simplicity we stick to $i=p-3$.  In this section $p$ is an odd prime, $p\geq 5$. 

The start point is the irreducibility of  $\ro_{g,p,\zeta, (i)}$, for any $i$, according to (the proof  given by)  
Roberts (\cite{R}, see also \cite{GM2}, Cor. 3.2).

Consider the topological closure $\mathcal G_{g,p,(i)}$ of  $\ro_{p, A_p^2, (i)}(
\widetilde{\Gamma_{g,1}})$. 
We know from \cite{F}, that when $g\geq 2$, $p\geq 7$ the group $\mathcal G_{g,p,(i)}$ is 
infinite and non-abelian. Denote by $V_{g,p,(i)}$ the representation of 
 $\mathcal G_{g,p,(i)}$ into $\mathbb U_{g,p,\zeta, (i)}$.
 
 If the representation $V_{g,p,(i)}$ were self-dual, its restriction to 
$\widetilde{\Gamma_{g-1,1}\times \Gamma_{1,1}}$ 
 would be a direct sum of self-dual and pairs of dual representations. The 
 invariant subspace $W_{g-1,(0)}\otimes W_{1,(0,i)}$ is  not self-dual 
 (see \cite{LW}, step 10), as $W_{g-1,(0)}$ is not self-dual.  Moreover, it is not dual to 
 $W_{g-1,(j)}\otimes W_{1,(j,i)}$, for any other values of $j$, since these subspaces 
are tensor products of irreducible representations and the corresponding dimensions 
of $W_{1,(j,i)}$ do not agree with that for $j=0$ unless $j=0$. 
  
We wish now to prove our claim by induction on $g$. When $g=2$ we 
choose $i=p-3$.  From above it follows that $\dim W_{2,p,(p-3)}=\frac{p^3-p}{24}$. 
Now the results from (\cite{LW}, section 4)  show that $\rho_{p,A_p^2,(p-3)}(\Gamma_g^1)$ is topologically dense 
into $P\mathbb U_{2,p,(p-3)}$. We can show using the same lines that the result holds for large enough $p$, for any $i$.

Further it follows from \cite{LW} that:
\begin{enumerate} 
\item  the restriction of $V_{g,p,(p-3)}$ to the identity component 
 $\mathcal G_{g,p,(p-3)}^{\circ}$  of $\mathcal G_{g,p,(p-3)}$ is  isotypic.  
 \item For any normal subgroup 
 $H\subset \mathcal G_{g,p,(p-3)}$ with the property that all  morphisms 
 $SL_2(\Z/p\Z)\to \mathcal G_{g,p,(p-3)}/H$ are trivial, the representation of $H$ into 
 $V_{g,p,(p-3)}$ is tensor indecomposable. 
 \item Moreover, $V_{g,p,(p-3)}$ is irreducible as a 
 $\mathcal G_{g,p,(p-3)}^{\circ}$-representation. 
\end{enumerate}

Further the content of Lemma \ref{growth} is the extension of (\cite{LW}, Lemma 12) to the case of surfaces with boundary. This is the main condition needed to prove the induction step from $g$ to $g+1$. 
It actually works for all $g\geq 3$, except for $g=2$. 

When $g=2$ we only obtain that  $\mathcal G_{3,p,(p-3)}^{\circ}$ is a simple compact Lie group of type $A_n$ and the representation $V_{3,p,(p-3)}$ is either the standard one or else the exterior or the symmetric square. 
In particular, if this representation were not the standard one, then $\dim W_{3, p, (p-3)}$ would be of the form $m(m+1)/2$, for some natural number $m\in \{n,n+1\}$. 
This situation could only occur for finitely many $p$, according to Lemma \ref{square}.

Eventually the arguments from (\cite{LW}, steps 14 and 15) show that 
the identity component $\mathcal G_{g,p,(p-3)}^{\circ}$ is a simple compact Lie group and for 
$g\geq 3, p\geq 7$  we have the equality 
$\mathcal G_{g,p,(p-3)}=S\mathbb U_{g,p,(p-3)}$.  
Thus  $\ro_{p, A_p^2, (p-3)}(\widetilde{\Gamma_{g,1}})$ is topologically dense into $S\mathbb U_{g,p,(p-3)}$. 
This implies that  $\ro_{p, \zeta, (p-3)}(\widetilde{\Gamma_{g}^{1}})$ is Zariski dense into  
$S\mathbb U_{g,p,\zeta, (p-3)}$ for all primitive roots of unity $\zeta$.

\subsection{Trace fields} 
Recall that $S\mathbb U_{g,p,(i)}$ is an absolutely almost simple simply connected algebraic group defined over $\Q(\zeta_p+\overline{\zeta_p})$ (i.e. its proper normal algebraic subgroups are finite). 
The adjoint trace field of a subgroup $\Delta\subset S\mathbb U_{g,p,(i)}$ is 
the field $\Q\left({\rm tr}(Ad(x)), x\in \Delta\right)$, where $Ad$ is the adjoint representation of 
$S\mathbb U_{g,p,(i)}$. We have the following extension of  the  corresponding result for mapping class groups of closed surfaces from (\cite{MRe}, section 4.3):
\begin{lemma}\label{tracefield}
Up to rescaling $\rho_{p,(p-3)}$ by some $2p$-th root of unity 
we can insure that the adjoint trace field of $\ro_{p,(p-3)}(\widetilde{\pi_g})$ is $\Q(\zeta_p+\overline{\zeta_p})$. 
\end{lemma}
\begin{proof}
If $\ell$ denotes the adjoint trace field in the statement then $\ell\subset \Q(\zeta_p+\overline{\zeta_p})$.  
The Zariski density and classical theorems of Vinberg (see \cite{MRe}, Prop.4.2) show that
$S\mathbb U_{g,p,(p-3)}$ is defined over $\ell$ and $Ad(\ro_{p,(p-3)}(\widetilde{\pi_g}))$ is contained in the group 
 $Ad(S\mathbb U_{g,p,(p-3)})(\ell)$  of $\ell$ points of  the adjoint group $Ad(S\mathbb U_{g,p,(p-3)})$. If we show that 
$\ro_{p,(p-3)}(\widetilde{\pi_g}))$ is contained in the group  $S\mathbb U_{g,p,(p-3)}(\ell)$ then the 
argument of (\cite{MRe}, section 4.3) will imply that  $\ell=\Q(\zeta_p+\overline{\zeta_p})$. 

If $Z$ is the center of $S\mathbb U_{g,p,(p-3)}$, then we have an exact sequence 
\[ Z(\Q(\zeta_p+\overline{\zeta_p}))\to S\mathbb U_{g,p,(p-3)}(\Q(\zeta_p+\overline{\zeta_p}))\to 
Ad(S\mathbb U_{g,p,(p-3)})(\Q(\zeta_p+\overline{\zeta_p}))\]

Let $\sigma\in Gal(\Q(\zeta_p+\overline{\zeta_p})/\ell)$. Then we have a homomorphism 
\[f:\ro_{p,(p-3)}(\widetilde{\pi_g})\to Z^1(Gal(\Q(\zeta_p+\overline{\zeta_p})/\ell), Z(\Q(\zeta_p+\overline{\zeta_p}))\]
\[ f(\gamma) (\sigma)=\gamma \sigma(\gamma^{-1}), \; {\rm for }\; \gamma\in \ro_{p,(p-3)}(\widetilde{\pi_g}), \sigma\in Gal(\Q(\zeta_p+\overline{\zeta_p})/\ell)\]
The group of 1-cocycles $Z^1$ is an abelian group and hence $f$ factors through the abelianization 
$H_1(\ro_{p,(p-3)}(\widetilde{\pi_g})$ which is a quotient of $(\Z/p\Z)^{2g}$. 
On the other hand  the group cohomology $H^1(H, Z)$ is killed by the order of the finite group $H$. 
Now,  the order of the group $Gal(\Q(\zeta_p+\overline{\zeta_p})/\ell)$ is a divisor of $\frac{p-1}{2}$. 
Thus elements in the image of the map induced by $f$:  
\[f_*:\ro_{p,(p-3)}(\widetilde{\pi_g})\to H^1(Gal(\Q(\zeta_p+\overline{\zeta_p})/\ell), Z(\Q(\zeta_p+\overline{\zeta_p}))\]
should be killed by both $p$ and some divisor of $\frac{p-1}{2}$ and hence they are trivial in cohomology. 
Therefore there exists $a$ in the center $Z(\Q(\zeta_p+\overline{\zeta_p}))$ such that 
\[ f(\gamma)(\sigma)=a \cdot \sigma(a^{-1})\]
and thus rescaling $\rho_{p,(p-3)}$ by $a$ will insure that $f(\gamma)$ is trivial for every $\gamma$ and 
hence $\ro_{p,(p-3)}(\widetilde{\pi_g}))$ is contained in the group  $S\mathbb U_{g,p,(p-3)}(\ell)$. 
Note now that the center $Z(\Q(\zeta_p+\overline{\zeta_p}))$ consists of scalars which are roots of unity in 
$\Q(\zeta_p+\overline{\zeta_p})$, and thus they are $2p$-th roots of unity. 
\end{proof}

\section{Proofs of the main theorems}
\subsection{Abundance of finite quotients}
We will need the following versions of the 
strong approximation theorem due to Nori-Weisfeiler. First, we record the statement due to Nori for 
algebraic groups defined over $\Q$: 
\begin{theorem}[\cite{No}, Thm.5.4]\label{Nori}
Let $G$ be a connected  linear algebraic group $G$ defined over $\Q$  
and $\Lambda\subset G(\Z)$ be a Zariski dense subgroup. 
Assume that $G(\C)$ is simply connected. Then the 
completion of $\Lambda$ with respect to the congruence  
topology induced from $G(\Z)$ is an open subgroup in 
the group $G(\widehat{\Z})$ of points of $G$ 
over the pro-finite completion $\widehat{\Z}$ of $\Z$.  
\end{theorem}
Further, in (\cite{MRe}, Thm. 2.6) Masbaum and Reid  stated the following consequence of 
the approximation theorem stated by Weisfeiler (\cite{Weis}, Thm. 10.5, Cor. 10.6), which is now valid for algebraic groups defined over number fields: 
\begin{theorem}[\cite{MRe}, Thm.2.6]\label{Weisfeiler}
If $\Delta\subset S\mathbb U_{g,p,(i)}(\Q(\zeta+\overline{\zeta}))$ is a Zariski dense subgroup of $S\mathbb U_{g,p,(i)}$ such that the adjoint trace field of $\Delta$ is 
$\Q(\zeta+\overline{\zeta})$, then for all but finitely many primes $\frak p$ in $\Q(\zeta+\overline{\zeta})$ 
the reduction homomorphism $\Delta \to S\mathbb U_{g,p,(i)}(\mathbf F_{\frak p})$ is surjective, where  
$\mathbf F_{\frak p}$ denotes the residue field  $\Q(\zeta+\overline{\zeta})/\frak p$. 
\end{theorem}

Our key ingredient in the proofs of the main theorems is the following result showing that infinitely many finite 
groups of Lie type should occur among the quotients of a Burnside-type group: 
\begin{proposition}\label{finitequotients}
Let $g\geq 2$ and  $p\equiv 3 ({\rm mod}\; 4)$, $p$ large enough prime.
Then, for all but finitely many primes $q$ there exist surjective homomorphisms 
$B(\widetilde{\pi_g},p,\mathcal S(S_g))\to \mathbb G_{g,p,(p-3)}(\Z/q^k\Z)$ and 
$B(\pi_g,p,\mathcal S(S_g))\to \mathbb PG_{g,p,(p-3)}(\Z/q^k\Z)$.  Moreover, for infinitely many $q$ 
the finite groups on the right hand side surject onto $PSL(N_{g,p}, \mathbf F_q)$, where $\mathbf F_q$ denotes the finite 
field on $q$ elements and $N_{g,p}=\dim W_{g,p,(p-3)}$. 
\end{proposition}
\begin{proof}
The linear algebraic group $G=\mathbb G_{g,p,(p-3)}$ satisfies 
the assumptions of Nori's Theorem \ref{Nori}.  
If we take $\Lambda$ to be a finite index subgroup of   
$\ro_{p,(p-3)}(\widetilde{\pi_g})$, then Theorem \ref{Nori} implies 
our claim for $k=1$. 

In fact $\ro_{p, (p-3)}|_{\widetilde{\pi_g}}$ factors through $B(\widetilde{\pi_g},p,\mathcal S(S_g))$, since each 
homotopy class of a simple closed curve on $S_g$ is sent into the composition of two commuting Dehn twists in $\widetilde{\Gamma_g^1}$. 
Moreover, the Dehn twist along the boundary curve $c=\partial S_g$ is central in $\Gamma_{g,1}$, 
and hence the image by $\ro_{g,p}$ of the center of $\widetilde{\Gamma_g^1}$ 
consists of central elements of finite order $p$. 

Then a classical result due to 
Serre  (see \cite{Serre}) for  $GL(2)$ 
and extended by Vasiu (see \cite{Vas}) to all reductive 
linear algebraic groups defined over $\Q$ improves the 
surjectivity statement to all $k\geq 1$.  

An alternate approach for $k=1$ would be to use directly the Nori-Weisfeiler approximation theorem on $S\mathbb U_{g,p,(p-3)}$. 
From Lemma \ref{tracefield} the group $\ro_{p,A_p^2,(p-3)}(\widetilde{\pi_g})\subset S\mathbb U_{g,p,(p-3)}$
has trace field $\Q(\zeta+\overline{\zeta})$, up to possibly translating it by a root of unity. 
Therefore, by Theorem \ref{Weisfeiler} for all but finitely many primes $\frak p$ in the trace field the 
reduction mod $\frak p$ is well-defined and provides a surjection $\ro_{p,A_p^2,(p-3)}(\widetilde{\pi_g})\to S\mathbb U_{g,p,(p-3)}(\mathbf F_{\frak p})$. According to the discussion in (\cite{Tits1}, p.55; \cite{PR}, 2.3.3) 
the group $S\mathbb U_{g,p,(p-3)}(\mathbf F_{\frak p})$ is either a special unitary group, when $\frak p$ is prime or ramified 
in $\Q(\zeta)$ or else a special linear group, when $\frak p$ splits completely in $\Q(\zeta)$. 
In particular, if $q$ is a rational prime which splits completely in $\Q(\zeta)$ and $\frak p$ a prime in 
$\Q(\zeta+\overline{\zeta})$ which divides $q$, then  
$S\mathbb U_{g,p,(p-3)}(\mathbf F_{\frak p})$ is isomorphic to $SL(N_{g,p}, \mathbf F_q)$, for all but finitely many $\frak p$.   
 \end{proof}

We now record  the following version of Hall's lemma (see \cite{Hall}) for finite groups, due to Dunfield and Thurston:
\begin{lemma}[\cite{DT}, Lemma 3.7]\label{Hall-finite}
Suppose that we have a set of epimorphisms $f_i:G\to H_i$, where $H_1,H_2,\ldots,H_k$ are 
non-abelian simple groups. If $f_i$ are pairwise non-equivalent, namely there is no isomorphism 
between $\alpha: H_i\to H_j$ such that $\alpha\circ f_i=f_j$, for $i\neq j$, then the map  
\[(f_1,f_2,\ldots,f_k):G\to H_1\times H_2 \times \cdots \times H_k\] 
is surjective.  
\end{lemma}

\begin{proposition}\label{corolar}
For large enough prime $p\equiv 3 ({\rm mod} \; 4)$ and 
$g\geq 2$ the group 
$\widehat{B(\pi_g,p,\mathcal S(S_g))}$ is neither finite-by-solvable, nor solvable-by-finite. 
\end{proposition}
\begin{proof}
For large enough prime $q$ the surjective maps $B(\pi_g,p,\mathcal S(S_g))\to \mathbb PG_{p,(p-3)}(\Z/q^k\Z)$ induce a continuous surjective homomorphism:
$\widehat{B(\pi_g,p,\mathcal S(S_g))}\to P\mathbb G_{p,(p-3)}(\Z_q)$. 

If $\widehat{B(\pi_g,p,\mathcal S(S_g))}$ had a prosolvable normal subgroup of finite index at most $N$, then 
$P\mathbb G_{g,p,(p-3)}(\Z_q)$ would also have a prosolvable normal subgroup of index at most $N$. 
But the index of the largest normal prosolvable group within $P\mathbb G_{g,p,(p-3)}(\Z_q)$  
goes to infinity with $q$.  It is well-known that  $P\mathbb G_{g,p,(p-3)}(\Z/q\Z)$ are 
finite simple groups (\cite{Tits1}, p.55; \cite{PR}, 2.3.3). More precisely, by Proposition \ref{finitequotients} 
we can find infinitely many finite groups of the form $PSL(N_{g,p}, \mathbf F_q)$ and $PU(N_{g,p},\mathbf F_q)$  among these quotients. 
In particular, a normal solvable subgroup of $P\mathbb  G_{g,p,(p-3)}(\Z_q)$ must project to the trivial 
subgroup of  $P\mathbb G_{g,p,(p-3)}(\Z/q\Z)$ and hence has index at least the size of the 
later. This is optimal, as 
\[P\mathbb G_{g,p,(p-3)}(q\Z_q)= \ker(P\mathbb G_{g,p,(p-3)}(\Z_q)\to  P\mathbb G_{g,p,(p-3)}(\Z/q\Z))\] 
is a pro-$q$ group and hence it is prosolvable. Now the size of  $P\mathbb G_{g,p,(p-3)}(\Z/q\Z)$ goes to infinity with $q$. 
Therefore $\widehat{B(\pi_g,p,M)}$ is not virtually prosolvable.

An alternate proof is as follows. 
Since   $P\mathbb G_{g,p,(p-3)}(\Z/q\Z)$ are simple non-abelian groups, for large $q$ they are 
pairwise non-isomorphic. From Hall's Lemma \ref{Hall-finite} we derive that the product homomorphism 
\[\widehat{B(\pi_g,p,\mathcal S(S_g))}\to \oplus_{q\geq m(p)}P\mathbb G_{g,p,(p-3)}(\Z/q\Z)\]  
is surjective. 

According to (\cite{RZ}, Corollary 4.2.4) a prosolvable group has all its finite quotients solvable. 
In our case any finite index normal subgroup of $\widehat{B(\pi_g,p,\mathcal S(S_g))}$ surjects onto infinitely many simple 
groups, and hence it cannot be virtually prosolvable. 

Eventually, suppose that  $\widehat{B(\pi_g,p,\mathcal S(S_g))}$ is solvable-by-finite, namely it 
contains a finite normal subgroup $L$ such that the quotient  $\widehat{B(\pi_g,p,\mathcal S(S_g))}/L$  is prosolvable. 
Then the image of $L$ in every large enough finite simple quotient $P\mathbb G_{g,p,(p-3)}(\Z/q\Z)$ should be 
trivial, as it cannot be the whole group by cardinality reasons. Therefore, $\widehat{B(\pi_g,p,\mathcal S(S_g))}/L$ 
surjects onto infinitely many finite simple groups  $P\mathbb G_{g,p,(p-3)}(\Z/q\Z)$. By the arguments above 
this contradicts the fact that  $\widehat{B(\pi_g,p,\mathcal S(S_g))}/L$  was supposed (virtually) prosolvable. 
\end{proof}

\subsection{Proof of Theorem \ref{prosolvable}} 
The case $m=1$ is settled in Proposition  \ref{corolar} above. 

Now, in order to prove a similar statement for  $B(\pi_g,p,\mathcal S_m(S_g))$, where $m \geq 2$ 
we have to pass to a finite cover of $S_g$. Indeed the  classes of closed immersed based loops in $S_g$ with no more than $m$ self-intersections up to a homeomorphism of $S_g$ form a finite set. Choose a set of based loops $M$ of representatives of this set. 

According to a classical Theorem of Scott (\cite{Sc1,Sc2}), given a based loop $\gamma$ on $S_g$ there exists 
a finite cover $S_h$ of $S_g$ and some $d$ such that the loop $\gamma^d$ lifts to an embedded loop in $S_h$. 
There exists then a finite characteristic cover, say of degree $d$, of pointed surfaces $f:(S_h,\tilde{z})\to (S_g,z)$ 
so that the $d$-th powers of all based loops from $M$ admit simple lifts based at $\tilde{z}$.
It follows that the $d$-th powers of  based loops from $\mathcal S_n(S_g)$ lift 
to simple based loops in $S_h$. 

Observe that the restriction of any automorphism of $\pi_g$ to the (image of) $\pi_h$, viewed as a subgroup, is 
an automorphism of $\pi_h$. This defines a homomorphism $F:\Gamma_g^1\to \Gamma_h^1$.
If $\varphi\in \Gamma_g^1$ is such that $\varphi(x)=x$, for any $x\in \pi_h$, then 
$\varphi(x^d)=x^d$, for any $x\in \pi_g$. Since surface groups are bi-orderable (this goes back to Magnus) 
we have $\varphi(x)=x$ for any $x\in \pi_g$, as a strict inequality for some $x$ would imply a strict inequality for its $d$-th powers.  Therefore $F$ is injective. 

Recall that for any based loop $\gamma$ on $S_g$ we have $f(f^{-1}(\gamma))=\gamma^d\in \pi_1(S_g,z)$, as the loop $\gamma$ is traveled $d$-times. If $\gamma\in \mathcal S_m(S_g)$, there exists some simple lift 
$\widetilde{\gamma}$ based at $\tilde{z}$. It follows that $f(\widetilde{\gamma})=\gamma^{m(\gamma)}\in \pi_1(S_g,z)$, where $m(\gamma)$ is a divisor of $d$. 

Denote by $ad_{S_g,\gamma}$ the action by conjugacy by $\gamma$, namely the image of 
$\gamma$ into $\Gamma_g^1=\Aut(\pi_g)$. 
As $\gamma^d$ belongs to the image of $\pi_h$ we can compute: 
\[ F(ad_{S_g, \gamma^d})=ad_{S_h, \widetilde{\gamma}^{d/m(\gamma)}}\] 
It follows that the image by $F$ of the group $\mathcal S_m(S_g)[nd]$ is contained into $\mathcal S(S_h)[n]$. 

Although $F(\pi_g)$ is not contained into $\pi_h$, it contains $\pi_h$ of finite index dividing $d$ since 
for any element $\gamma\in \pi_g$ its image $F(ad_{S_g, \gamma})^d\in \pi_h$. 

Further the map $F$ induces a homomorphism 
\[\overline{F}: B(\pi_g,nd,\mathcal S_m(S_g))\to \Gamma_h^1/F(\mathcal S_m(S_g)[nd])\]
Now, the subgroup $\pi_h/F(\mathcal S_m(S_g)[nd])$ is of finite index into the image 
$\overline{F}(B(\pi_g,nd,\mathcal S_m(S_g)))$.  As 
\[F(\mathcal S_m(S_g)[nd])\subset \mathcal S(S_h)[n]\subset \pi_h\]
and $\pi_h$ is a normal subgroup in $\Gamma_h^1$, the group 
$\pi_h/F(\mathcal S_m(S_g)[nd])$ surjects onto the Burnside-type group 
$B(\pi_h, n, \mathcal S(S_h))$. 

It follows that the group $B(\pi_g,nd,\mathcal S_m(S_g))$ has a finite index subgroup which surjects onto 
$B(\pi_h, n, \mathcal S(S_h))$, and in particular it is not virtually prosolvable nor solvable-by-finite or finite-by-solvable.

\begin{remark}
If we had proven that  $d=1$ is convenient for all $m$  
then the family of finite quotients of $B(\pi_g,nd,\mathcal S_m(S_g))$ would provide a negative answer to  
Problem 4' from \cite{Z}. 
\end{remark}

\begin{remark}
It would be interesting to know whether the image of  
$\widehat{B(\pi_g,p,M)}\to \prod_{q\geq m(p)}\mathbb G_{p,(i)}(\Z_q)$ is 
open.  
\end{remark}

\begin{remark}
The  arithmetic group $\mathbb G_{p,(p-3)}(\Z)$, for $g\geq 3$ and  prime 
$p\geq 5$ has the congruence property.  
This follows from results of Tomanov (see \cite{Tom}, Main Thm. (a)) 
and Prasad and Rapinchuk (see \cite{PrR}, Thm. 2.(1) and Thm. 3)
on the congruence kernel for $\Q$-anisotropic algebraic groups of 
type $^2A_{n-1}$, with $n\geq 4$. 
Moreover, $\mathbb G_{p, (p-3)}(\Z)$ 
is cocompact in $\mathbb G_{p,(p-3)}(\R)$, since 
it is $\Q$-anisotropic, by a classical result of Borel and Harish-Chandra 
(see \cite{BH}).    
\end{remark}

\subsection{Proof of Corollary \ref{ubiquity}}
We have  the following sequence of inclusions:
\[\rho_{p,(p-3)}(\pi_g)\subset \rho_{p,(p-3)}(\Gamma)\subset \rho_{p,(p-3)}(\Gamma_g^1)\subset P\mathbb G_{g,p,(p-3)}(\Z)\]
Proposition \ref{finitequotients} shows that for $g\geq 2$ and large enough prime 
$p\equiv 3 ({\rm mod}\; 4)$  the reduction mod $q^k$  sends $\rho_{p,(p-3)}(\pi_g)$  
onto $P \mathbb G_{g,p,(p-3)}(\Z/q^k\Z)$, 
for all but finitely many primes $q$. Therefore,  all groups in the sequence above have the same image 
$P\mathbb G_{g,p,(p-3)}(\Z/q^k\Z)$ under the reduction mod $q^k$.

For infinitely many $q$ the groups 
$P\mathbb G_{g,p,(p-3)}(\Z/q\Z)$ are either of the form $PSL(N_{g,p}, \mathbf F_q)$ or $PSU(N_{g,p},\mathbf F_q)$, for some $N_{g,p}$ going to infinity with $p$. 
This gives the first assertion of Corollary \ref{ubiquity}.

Eventually, the fundamental group $\pi_1(M^3)$ of the 
fibered 3-manifold $M^3$ with monodromy $\varphi\in {\rm Aut}^+(\pi_g)$ is isomorphic to the semi-direct product 
$\pi_g\rtimes_{\varphi} \Z$, where the action of the generator of $\Z$ on $\pi_g$ is given by $\varphi$.   
Now, $\pi_1(M^3)$ embedds in $\Gamma_g^1$, since it is isomorphic to the 
preimage of the group $\langle \overline{\varphi}\rangle \subset \Gamma_g$ generated by 
the class $\overline{\varphi}\in \Gamma_g$ of $\varphi$, under the homomorphism $\Gamma_g^1\to \Gamma_g$. 
Then the claim follows from above.

\begin{remark}\label{faithful}
Quantum representations are asymptotically faithful (see \cite{A,FWW}); for a surface 
of genus $g\geq 2$ with one boundary component and an infinite set $A$ of odd numbers we have  
(see \cite{FWW}, Thm.3.3):
\[ \cap_{p\in A, i\in \mathcal C_p}\ker \rho_{g,p, (i)} =1\in \Gamma_g^1\]
It seems that (\cite{M}) the methods of \cite{MN,MC} could improve the asymptotic faithfulness above to the case where
we only consider a  single boundary color  for each $p$, namely that: 
\[ \cap_{p\in A} \ker \rho_{g,p,(f(p))}=1\in \Gamma_g^1, \; {\rm provided \; that} \; \lim_{p\to \infty}\frac{f(p)}{p}=1\]
Now, a non-trivial element of 
$P\mathbb G_{g,p,(p-3)}(\Z)$ can be detected by reduction modulo some prime $q$ belonging to any given infinite set of 
primes. Thus projections onto finite simple quotients $P \mathbb G_{g,p,(p-3)}(\Z/q\Z)$ could detect any 
non-trivial element of $\Gamma_g^1$, when $p$ and $q$ belong to  (any) infinite sets of primes. 
\end{remark}

\subsection{Proof of Theorem \ref{Lubotzky}}  
Consider  the  homomorphism $\Psi_{g,p,q}$ obtained by composing the 
projection on a simple factor, the reduction mod $q^k$  and $\rho_{p,(p-3)}$ as follows:
\[ \Gamma_g^1\to \rho_{p,(p-3)}(\Gamma_g^1)\subset P\mathbb G_{g,p,(p-3)}(\Z) \to 
   P\mathbb G_{g,p,(p-3)}(\Z/q\Z)\]
Recall from the first lines of the proof of Corollary \ref{ubiquity} that for any large prime $p\equiv 3 \; ({\rm mod} \; 4)$  
and large enough prime $q$  we have:
\[ \Psi_{g,p,q}(\pi_g)=\Psi_{g,p,q}(\Gamma_g^1)=P\mathbb G_{g,p,(p-3)}(\Z/q\Z)\]

We obtained therefore infinitely many surjective homomorphisms $\Psi:{\rm Aut}^+(\pi_g)\to F$ onto finite simple groups 
$F$ of Lie type, with the property that $\Psi(\pi_g)=F$ (we dropped the subscripts to simplify the notation). 
We now claim that $\ker\Psi|_{\pi_g}$ are characteristic subgroups, which will settle the result. 

Observe first that every subgroup $\ker\Psi|_{\pi_g}\subset \pi_g$ is  ${\rm Aut}^+(\pi_g)$-invariant because:
\[\Psi|_{\pi_g}(\varphi(x))=\Psi(\varphi x \varphi ^{-1})=1, {\rm for} \; x\in \ker\Psi|_{\pi_g}, \varphi\in {\rm Aut}^+(\pi_g)\]

We are almost done since  ${\rm Aut}^+(\pi_g)$ is an index 2 subgroup of  the group of all automorphisms 
${\rm Aut}(\pi_g)$. 
To proceed further, realize  $S_g$ in the Euclidean space as the double  along the boundary of $S_{\frac{g}{2},1}$, if $g$ is even and  
$S_{\frac{g-1}{2},2}$, when $g$ is odd, respectively. Let $\tau$ denote the Euclidean symmetry 
exchanging the two halves of $S_g$, so that $\tau$ is an involution reversing the orientation of $S_g$. 
We denote by the same letter the corresponding mapping class $\tau\in {\rm Aut}(\pi_g)$.  Further,  
$ {\rm Aut}(\pi_g)$ is generated by $ {\rm Aut}^+(\pi_g)$ and an arbitrary orientation reversing mapping class, in particular $\tau$.

We have now the following lemma whose proof is postponed a few lines: 
\begin{lemma}\label{reversing}
The homeomorphism $\tau$ induces an anti-linear map 
$\tau_*:W_{g,p,(i)}\to W_{g,p,(i)}$ which  coincides with the anti-linear involution $J$ induced by 
the complex conjugation of coordinates.  
\end{lemma}

The group of homeomorphisms of $\Sigma_{g,1}$ which are identity on the boundary  
acts on the space of conformal blocks $W_{g,p,(i)}$, since their construction is functorial. This action  provides a representation of  ${\rm Aut}(\pi_g)$ into the general linear group of the real 
vector space underlying $W_{g,p,(i)}$ which extends $\rho_{p,\zeta,(i)}$. We keep the same notation $\rho_{p,\zeta,(i)}$
for this extension.

Now, Lemma \ref{reversing} gives us: 
\[ \rho_{p,\zeta,(i)}(\tau x\tau^{-1})=J\rho_{p,\zeta,(i)} (x)J = \rho_{p, \overline{\zeta},(i)}(x), \; {\rm for} \; x \in {\rm Aut}^+(\pi_g)\]   
Further, if we identify  $\tau(x)\in\pi_g$  with its image by the point pushing 
map $\pi_g \to \Gamma_g^1$ arising in the Birman exact sequence, we can write: 
\[ \tau(x)=\tau x \tau^{-1}, \; {\rm for } \; x\in \pi_g\]
We then obtain from above:
\[ \rho_{p, \zeta, (i)}(\tau(x))=\rho_{p,\zeta,(i)}(\tau x \tau^{-1})= J \rho_{p,\zeta, (i)}(x) J= \rho_{p,\overline{\zeta},(i)}(x), \; {\rm for} \; x\in \pi_g\]
Therefore, $\ker(\rho_{p,(i)}|_{\pi_g})$ is $\tau$-invariant and hence 
a characteristic subgroup of $\pi_g$.  

The complex conjugation $J$ induces an automorphism of $\mathbb G_{g,p,(i)}$ 
keeping invariant the lattice  $\mathbb G_{g,p,(i)}(\Z)$.  Then the two surjective homomorphisms $\Psi|_{\pi_g}$ and $\Psi|_{\pi_g}\circ \tau:\pi_g\to F$ are equivalent, so that $\ker(\Psi|_{\pi_g})$ is both ${\rm Aut}^+(\pi_g)$-invariant and $\tau$-invariant and hence a characteristic subgroup, as claimed. 
This proves Theorem \ref{Lubotzky}.  

\begin{proof}[Proof of Lemma \ref{reversing}]
The homeomorphism $\tau$ extends to $S^3$ and sends any link into its mirror image. 
Recall that the Kauffman bracket invariant $\langle \; \rangle_A$ at the parameter 
$A$ is well-behaved with respect to the mirror symmetry, namely we have:  
\[ \langle \tau(K) \rangle_A = \langle K \rangle _{A^{-1}}\]
Therefore $\tau$ induces a map at the level of skein modules of the handlebodies $H_g$ (and $\overline{H}_g$) 
still denoted $\tau:S_A(H_g)\to S_{A^{-1}}(H_g)$. 
According to the definition of the sesquilinear form $\langle , \rangle$ we have:  
\[ \langle \tau(x), \tau(y)\rangle_{A^{-1}} = \langle \tau(x\sqcup y)\rangle_{A^{-1}}=\langle x \sqcup y\rangle_{A^{-1}}\]
It follows that $x\in \ker \langle , \rangle _A$ if and only if $\tau(x)\in \ker \langle , \rangle _{A^{-1}}$.  
We obtain $W_{g,p,(i)}$ as the quotient by the kernel of $\langle , \rangle_{A}$ after specifying 
an embedding $\Z[A,A^{-1}]\to \C$  sending $A$ to a $2p$-th root of unity. It follows that $\tau$ induces 
the complex conjugation at the level of $W_{g,p,(i)}$. In other terms $\tau$ provides an isomorphism between 
the space of conformal blocks  $W_{g,p,(i)}$ associated to the surface $\Sigma_{g, (i)}$ and its dual 
$W_{g,p,(i)}^*$ which is associated to the surface $-\Sigma_{g, (i)}$ with the opposite orientation. 
 \end{proof}

\begin{remark}
More generally, kernels of unitary TQFT representations are characteristic subgroups of $\pi_g$. 
\end{remark}

\subsection{Generalized congruence quotients}
A {\em principal congruence subgroup} of $\Gamma_g^1$ is the kernel of the homomorphism 
${\rm Aut}^+(\pi_g)\to {\rm Aut}(F)$ induced by some surjective homomorphism $\pi_g\to F$ onto 
a finite characteristic quotient $F$ of $\pi_g$. We actually only need an ${\rm Aut}^+(\pi_g)$-invariant quotient $F$. 
The image of  ${\rm Aut}^+(\pi_g)$ within ${\rm Aut}(F)$ is called a {\em principal congruence quotient}. 
This construction naturally extends to all characteristic quotients $F$ of $\pi_g$, not necessarily finite ones; we will call them 
{\em generalized principal congruence subgroups} and {\em quotients} respectively.

\begin{proposition}\label{congruence}
If $\rho_{p,(i)}|_{\pi_g}:\pi_g\to P\mathbb G_{g,p,(i)}$ is Zariski dense then 
the mapping class group representation $\rho_{p,(i)}$ is equivalent to a generalized principal 
congruence representation. In particular, this holds when $i=p-3$. 
\end{proposition}
\begin{proof}
The action by conjugacy of $\rho_{p,(i)}(\Gamma_g^1)$ on $\rho_{p,(i)}(\pi_g)$ provides a homomorphism: 
\[ \Phi: \rho_{p,(i)}(\Gamma_g^1)\to {\rm Aut}(\rho_{p,(i)}(\pi_g))\]
In the proof of Lemma \ref{reversing} we noted that $\rho_{p,(i)}(\pi_g)$ are characteristic quotients of $\pi_g$. 
It remains to prove that the homomorphism $\Phi$ is injective. Let $\varphi\in \Gamma_g^1$ such that 
$\rho_{p,(i)}(\varphi)\in \ker \Phi$. This is equivalent to:
\[ \rho_{p,(i)}(\varphi)\rho_{p,(i)}(x)\rho_{p,(i)}(\varphi)^{-1}=\rho_{p,(i)}(x), \; {\rm for \; all} \; x\in \pi_g\]
Since $\rho_{p,(i)}(\pi_g)$ is Zariski dense in $P\mathbb G_{g,p,(i)}$ this implies that 
$\rho_{p,(i)}(\varphi)$ is identity. This proves the claim. 
\end{proof}
\begin{remark}
The same proof works under the weaker assumption that  $Ad\circ \rho_{p,(i)}|_{\pi_g}$ is irreducible (see \cite{Sik}, Prop.14). 
Moreover, if  we suppose that $\rho_{p,(i)}|_{\pi_g}$ is irreducible, then the homomorphism $\Phi$ has a finite abelian kernel 
(see \cite{Sik}, Prop.15).   
\end{remark}
\begin{proposition}\label{finite-congruence}
The finite quotients  $P\mathbb G_{g,p,(p-3)}(\Z/q^k\Z)$  of $\Gamma_g^1$ obtained in Proposition \ref{finitequotients}
are principal congruence quotients. 
\end{proposition}
\begin{proof}
The proof of Theorem \ref{Lubotzky} and Proposition \ref{congruence} above show 
that the image of $\Gamma_g^1$ acts by conjugacy on the finite characteristic quotient of $\pi_g$ 
obtained by reduction mod $q^k$ of $\rho_{p,(p-3)}(\pi_g)$. 
The reduction mod $q^k$ of the map 
$\Phi$ above is still injective since  $P\mathbb G_{g,p,(p-3)}(\Z/q\Z)$ are center-free. 
\end{proof}

\subsection{Comments}
The statement of Theorem \ref{Lubotzky} had several other reformulations which were discussed in (\cite{Lub}, 6.4).  
Conjecture 6.12 from \cite{Lub} claims that for all finite dimensional linear representations of the group 
${\rm Aut}(\mathbb F_n)$ of automorphisms of the free group $\mathbb F_n$ on $n\geq 3$ generators 
the image of the subgroup $\mathbb F_n$ of inner automorphisms is virtually solvable. 
Proposition \ref{density} above shows that a similar statement cannot hold when the free group $\mathbb F_n$ is replaced by 
a surface group $\pi_g$, $g\geq 2$.  This already follows from (\cite{KS}, Cor. 4.2).  Although representations $\ro_{g,p,(i)}$ arising here are all projective representations, they can be easily converted into linear 
representations with the same properties by considering  the tensor product $\ro_{g,p,(i)}\otimes \ro_{g,p,(i)}^*$ with their respective dual representations.  
This conjecture is both related to the non-linearity of ${\rm Aut}(\mathbb F_n)$, for $n\geq 3$ following Formanek and Procesi 
and the Weigold conjecture. The later states that ${\rm Aut}(\mathbb F_n)$ acts transitively on the set of kernels 
of surjective homomorphisms $\mathbb F_n\to G$, for any finite simple group $G$ (see \cite{Lub}) and it is known to hold for large enough $n$ in terms of $G$.  

In (\cite{Lub}, section 10) the authors discussed similar questions for a surface group. 
Any homomorphism  $\pi_g\to G$ defines a class in $H_2(G)$ which is the image of the fundamental class of $H_2(\pi_g)$. 
This class is left invariant by the left composition with automorphisms from ${\rm Aut}(\pi_g)$ and its image 
in $H_2(G)/{\rm Out}(G)$ is also invariant by the right composition with automorphisms from ${\rm Aut}(G)$. 
The extended Weigold conjecture asks whether ${\rm Aut}(\pi_g)$ (or just ${\rm Aut}^+(\pi_g)$)
acts transitively on the set of kernels of surjective homomorphisms  $\pi_g\to G$  
corresponding to a given class in $H_2(G)/{\rm Out}(G)$, provided $G$ is a finite simple group and $g\geq 3$. 
This was proved to hold for large enough $g$, depending on $G$ in \cite{DT}.

The existence of characteristic finite simple quotients of $\pi_g$, $g\geq 2$ from Theorem \ref{Lubotzky} and the 
discussion in (\cite{Lub}, 6.4) also show that this  analog of the Weigold conjecture for surface groups does not hold. 
We choose $p$ such that $N_{g,p}$ are odd and then primes $q$ such that $q-1$ is coprime to $N_{g,p}$. Then $H_2(PSL(N_{g,p}, \mathbf F_q))=1$, so that there is only one class of homomorphisms. One knows from \cite{DT} 
that there exists a large orbit of ${\rm Aut}(\pi_g)$ where this group acts as an alternating group. 
On the other hand for infinitely many  finite quotients  of $\pi_g$ of the form 
$PSL(N_{g,p}, \mathbf F_q)$ the action of ${\rm Aut}(\pi_g)$ is trivial. Thus for large enough $q$ there exist several 
${\rm Aut}(\pi_g)$-orbits (in the same class). 

Another problem stated in (\cite{Lub}, 6.4) for the free group is whether, for a Chevalley group scheme $F$ 
and surjective homomorphisms $\Psi: \pi_g\to F(\mathbf F_q)$ the 
number of conjugacy classes in the image $\Psi(\mathcal S(S_g))$ 
of primitive elements (i.e. simple closed curves) is unbounded as $q\to \infty$. 
For those $\Psi$ encountered in the proof of Theorem \ref{Lubotzky}
the image of $\Gamma_g^1$ is still $F(\mathbf F_q)$ and hence there are at most $\left[\frac{g}{2}\right]+1$ conjugacy classes for all $q$.

{
\small      
      
\bibliographystyle{plain}

}

\end{document}